\documentclass[preprint,10pt]{elsarticle}
\usepackage{amsmath,amsopn,amssymb,epsfig,amsthm}%,hyperref}
\usepackage{color}
\allowdisplaybreaks

\newcommand{\bb}{\begin{bmatrix}}
\newcommand{\eb}{\end{bmatrix}}
\numberwithin{equation}{section}
\definecolor{amethyst}{rgb}{0.6, 0.4, 0.8}
\definecolor{orange}{rgb}{1,0.5,0}
\newtheorem{Algorithm}{Algorithm}[section]
\newtheorem{Theorem}{Theorem}[section]
\newtheorem{Lemma}{Lemma}[section]

\newtheorem{example}{Example}[section]
\newtheorem{Definition}{Definition}[section]

\def\lp{{\limsup\limits_{k\rightarrow\infty}}}

\catcode`@=11
\newbox\@temp
\def \blap#1{\vbox to 0mm{#1\vss}}

\newcommand{\bset}[3][2mm]{#3\llap{%
      \blap{\vskip#1\hbox to 0mm{\hss $#2$\hss}}%
      \setbox\@temp\hbox{\mbox{$#3$}}\hskip0.5\wd\@temp}}
\def\hhline{%
  \noalign{\ifnum0=`}\fi\hrule \@height 3\arrayrulewidth \futurelet
   \@tempa\@xhline}
\catcode`@=12
\begin{document}
\begin{frontmatter}

\title{
On the semigroup property for some structured iterations
}

\author{Matthew M. Lin \corref{cor1}\fnref{fn2}}
\ead{mhlin@mail.ncku.edu.tw}
\address{Department of Mathematics, National Cheng Kung University, Tainan 701, Taiwan.}
\author{Chun-Yueh Chiang\corref{cor2}\fnref{fn1}}
\ead{chiang@nfu.edu.tw}
\address{Center for General Education, National Formosa
University, Huwei 632, Taiwan.}

\cortext[cor2]{Corresponding author}
\fntext[fn1]{The first
author was supported by the Ministry of Science and Technology of Taiwan under grant 107-2115-M-006-007-MY2.}
\fntext[fn2]{ The second author was supported by the Ministry of Science and Technology of Taiwan under grant 107-2115-M-150-002.}

\date{ }
\begin{abstract}

Nonlinear matrix equations play a crucial role in science and engineering problems. However, solutions of nonlinear matrix equations cannot, in general, be given analytically. One standard way of solving nonlinear matrix equations is to apply the fixed-point iteration with usually only the linear convergence rate.  To advance the existing methods, we exploit in this work one type of semigroup property and use this property to propose a technique for solving the equations with the speed of convergence of any desired order. We realize our way by starting with examples of solving the scalar equations and, also, connect this method with some well-known equations including, but not limited to, the Stein matrix equation, the generalized eigenvalue problem, the generalized nonlinear matrix equation, the discrete-time algebraic Riccati equations to express the capacity of this method.

\end{abstract}

\begin{keyword}
Nonlinear matrix equations,\,Sherman Morrison Woodbury formula,\,Semigroup property,\,Iterative methods,\,Acceleration methods,\,{R}-superlinear convergence
% Q-superlinear
\MSC 65F10 \sep  65H05 \sep  15A24\sep 15A86%\sep 65F10\sep 65F05
\end{keyword}
\end{frontmatter}

\pagestyle{myheadings} \thispagestyle{plain}
\section{Introduction}
Finding one or
more roots of a nonlinear matrix equation $F(X)=0$ over the filed $\mathbb{C}^{n\times n}$ is one of the more commonly occurring problems of
applied mathematics. In most cases explicit solutions are not available, finding the roots often require iterative methods, for example,
%Two classical methods are useful:
\begin{itemize}
\item[1.] the fixed-point iteration:
 \begin{align}\label{fixiter}
{X_{k+1}=G(X_k)}, \quad X_1 \mbox{ is given,}
\end{align}
where $G(X):= X-F(X)$.
\item[2.] the classic
Newton's iteration:
\begin{align}\label{NW}
{X_{k+1}=X_k-(F_{X_k}^{'})^{-1}(F(X_k))},\quad X_1 \mbox{ is given,}
\end{align}
where $F_{Z}^{'}$ denoted the Fr$\rm\acute{e}$chet derivative of $F$ at $Z$.
\end{itemize}
It is known that the fixed-point iteration~\eqref{fixiter} requires less computational costs, but its convergence speed is frequently linear. On the contrary, the classic Newton's iteration~\eqref{NW} under certain conditions has local quadratic convergence but requires more computational costs~\cite{Kelley95}.

%Recently, authors in~\cite{MR2156099, MR777559} investigate a class of higher order Newton-type iterations:
%\[
%{X_{k+1}=X_k-\sum\limits_{k=1}^\infty a_k[\mathcal{L}_F(X_k)]^{k-1}(F_{X_k}^{'})^{-1}(F(X_k))},\quad X_1 \mbox{ is given,}
%\]
%where $F_{X_k}^{'}$ is the Fr$\rm\acute{e}$chet derivative of $F$ at $X_k$,
%the sequence $\{a_k\}$, satisfying $\sum\limits_{k=1}^\infty a_k t^k<\infty$ for some $t\neq 0$,
% is a non-increasing real sequence of nonnegative numbers with $a_1=1$ and $a_2=\frac{1}{2}$, and $\mathcal{L}_F(t)$ is the so-called ``degree of logarithmic convexity''\cite{MR1686356}. Specifically, in the scalar case, the  operation $\mathcal{L}_F(t)$ is given by
%$\mathcal{L}_F(t)=\dfrac{F(t)F^{''}(t)}{(F^{'}(t))^2}$~\cite{MR777559}.

%
In this work, we investigate under what circumstance, it can accelerate the fixed-point iteration. We use the following definition for the rate of convergence to evaluate the speed at which a convergent sequence reaches its limit. See
~\cite{doi:10.1137/1.9780898719468,Kelley95,Potra2001,Bini2012} and the references therein.

\begin{Definition}\label{def:cov}
Let $r>1$ be an integer.
Given a sequence $\{X_k\}\subseteq\mathbb{C}^{n\times n}$ and an induced matrix norm $\|.\|$, then
\begin{itemize}
  \item [1.] $X_k$ converges \rm{R}-linearly to $X_*$ if
  %there exists a number $\sigma \in (0,1)$ %$c>0$ and
  %$0<\sigma<1$
\begin{align}\label{q1}
\limsup\limits_{k\rightarrow\infty}\sqrt[k]{\|X_k -  X_*\|} \leq \sigma, \quad \sigma \in (0,1),
\end{align}
%for sufficiently large $k$,
and $X_k$ converges \rm{R}-superlinearly to $X_*$ with order $r$ if
%there exists  a number $\sigma \in (0,1)$ %$c>0$ and
%$0<\sigma<1$
\begin{align}\label{q2}
\limsup\limits_{k\rightarrow\infty}\sqrt[r^k]{\|X_k -  X_*\|}\leq\sigma, \quad \sigma \in (0,1).
\end{align}
%for sufficient large $k$.
%

%

 \item [2.] $X_k$ converges \rm{Q}-linearly to $X^*$ if %there exists $d>0$ such that
\begin{align}\label{q3}
\limsup\limits_{k\rightarrow\infty}\dfrac{\|X_{k+1}-X_*\|}{\|X_k-X_*\|} = \sigma, \quad \sigma \in (0,1),
\end{align}
%for sufficient large $k$,
and $X_k$ converges \rm{Q}-superlinearly to $X_*$ with order $r$ if $\lim\limits_{k\rightarrow\infty}X_k=X_*$ and
\begin{align}\label{q4}
\limsup\limits_{k\rightarrow\infty}\dfrac{\|X_{k+1}-X_*\|}{\|X_k-X_*\|^r} = \sigma\in\mathbb{R}, \quad \sigma>0.
\end{align}
%for sufficient large $k$.
\end{itemize}
\end{Definition}
%
%
%\begin{Remark}
%Note that inequalities \eqref{q1}, \eqref{q2}, \eqref{q3} and \eqref{q4} are respectively equivalent to
%%
%\begin{eqnarray*}
%&{\limsup\limits_{k\rightarrow\infty}\sqrt[k]{\|X_k -  X^*\|} \leq \sigma, \,
%\limsup\limits_{k\rightarrow\infty}\sqrt[r^k]{\|X_k -  X^*\|}\leq\sigma,}\,\\
%&\limsup\limits_{k\rightarrow\infty}\dfrac{\|X_{k+1}-X^*\|}{\|X_k-X^*\|}<\infty,
%\,
%\limsup\limits_{k\rightarrow\infty}\dfrac{\|X_{k+1}-X^*\|}{\|X_k-X^*\|^r}<\infty.
%\end{eqnarray*}
{In particular,  we say that %the quantity $\sigma$ is the convergence rate of this sequence and
$X_k$ converges R-quadratically, R-cubically, and   Q-sublinearly to $X_*$ if %the conditions:
\[
 r=2 \mbox{ in \eqref{q2}, } r=3 \mbox{ in \eqref{q2}}, \mbox{ and } \lim\limits_{k\rightarrow\infty}X_k=X_*, \sigma=1 \mbox{ in \eqref{q3}, }
\]
respectively, are further satisfied.}

For a direct interpretation of methods with high order convergence,
see authors in~\cite{MR2156099, MR777559} investigate a class of higher order Newton-type iterations:
\[
{X_{k+1}=X_k-\sum\limits_{k=1}^\infty a_k[\mathcal{L}_F(X_k)]^{k-1}(F_{X_k}^{'})^{-1}(F(X_k))},\quad X_1 \mbox{ is given,}
\]
where $F_{X_k}^{'}$ is the Fr$\rm\acute{e}$chet derivative of $F$ at $X_k$,
the sequence $\{a_k\}$, satisfying $\sum\limits_{k=1}^\infty a_k t^k<\infty$ for some $t\neq 0$,
 is a non-increasing real sequence of nonnegative numbers with $a_1=1$ and $a_2=\frac{1}{2}$, and $\mathcal{L}_F(t)$ is the so-called ``degree of logarithmic convexity''\cite{MR1686356}. Note that in the scalar case~\cite{MR2156099,MR3144340,MR3293985},
 the operation $\mathcal{L}_F(t)$ is given by $\mathcal{L}_F(t)=\dfrac{F(t)F^{''}(t)}{(F^{'}(t))^2}$~\cite{MR777559},
 and
%  See the following Newton-type methods in the scalar case for examples~\cite{MR777559,MR2156099,MR3144340,MR3293985}:
\begin{itemize}
\item[1.] if $a_k = 0$ for $k >2$, we have the Chebyshev method with the constructed sequence converging R-superlinearly to the solution with order $r=3$,

\item [2.]  if $a_k=\dfrac{1}{2^{k-1}}$ for $k> 2$, we have the Halley method  with the constructed sequence converging R-superlinearly to the solution with order $r=3$,

 \item[3.] if $a_k=(-1)^{k-1}\binom{-\frac{1}{2}}{k-1}$ for $k>2$, we have the Ostrowski method with the constructed sequence converging R-superlinearly to the solution with order $r=4$.

\end{itemize}
%In particular, in the scalar case, the operation $\mathcal{L}_F(t)$ is given by $\mathcal{L}_F(t)=\dfrac{F(t)F^{''}(t)}{(F^{'}(t))^2}$~\cite{MR777559}.

This implies that Newton-type methods generally work better with high order convergence, when applied to find non-repeated roots of a differentiable function. However,  these methods require the computation of the inverse of Fr$\rm\acute{e}$chet derivatives of the objective function. This indicates that the Newton-type methods can only be used when the objective functions are differentiable. In particular, there is no guarantee that  Newton-type methods will converge if the selected starting values are too far from the exact solutions. To this end, once the objective function is not differentiable, we look for an accelerated approach to speed up the iterations so that the order of the rate of the convergence can be as high as possible.

More precisely, suppose $\{X_i\}$ be a sequence generated by the fixed-point iteration. We would like to develop a method which can obtain a subsequence of $\{X_i\}$, say $\{X_{i_j}\}$ with $i_j\geq i$, so that the subsequence $\{X_{i_j}\}$ converges to the desired solution more rapidly than the sequence $\{X_i\}$ does.
One way to do that is
based on the original fixed-point iteration~\eqref{fixiter} to build up an iteration:
\begin{equation}\label{fixiter2}
Y_{k+1}=G_r(Y_k)
\end{equation}
with $Y_1:=X_1$ such that for a positive integer $r>1$, $Y_k=X_{r^{k-1}}$ for each $k\geq 1$.  Namely, the operation $G_r$ construct a new relationship between $X_{r^{k}}$ and $X_{r^{k-1}}$.  Though this presented idea is simple, but the outcome result after the accelerated process is fruitful and can be elucidated as below:

\begin{itemize}
\item[1.] iterations from~\eqref{fixiter}
\[
X_1\rightarrow X_2 \rightarrow X_3\cdots \rightarrow X_k
\rightarrow\cdots,
\]
\item[2.] iterations from~\eqref{fixiter2}
\[
X_1\rightarrow X_r \rightarrow X_r^2\cdots \rightarrow X_{r^{k-1}}
\rightarrow\cdots.
\]
\end{itemize}
However, how to make sure the existence of the iterated operation $G_r$ is the primary goal
of our work. We begin with our investigation through the observation of the following example.

\begin{example}\label{simit}
Consider a linear scalar equation defined by
\begin{equation}\label{SI}
x=ax+b,
\end{equation}
where $a,b\in \mathbb{C}$ with $|a|<1$.
This solution $x^* = \frac{b}{1-a}$ can be found without any difficulty. Here, we see that how the fixed-point iteration, can be applied and accelerated to calculate the solution $x^*$. That is, given an initial value $x_1$, iterate
\begin{equation*}
x_{k+1}=a x_{k}+b, \quad k=1,\ldots
\end{equation*}
until convergence.
Observe that the $r^{k}$-th iterative step can be expressed in items of the $r^{k-1}$-th iterative step as follows:
\begin{align*}
x_{r^k} =a_k{x_{r^{k-1}}}+b_k,
\end{align*}
where ${a}_k = a^{r^{k-1}(r-1)}$ and
${b}_{k}= b\dfrac{1-a^{r^{k-1}(r-1)}}{1-a}$,
 for $k\geq 1$.
Let %$a_k = a^{k-1}$,
$\hat{x}_k = x_{r^{k-1}}$.  We can see that the sequence $\{x_k\}$ and
 the sequence
 $\{\hat{x}_k\}$, given by
 \begin{subequations}\label{doubit}
 \begin{align}
\hat{x}_{k+1} &={a}_{k}\hat{x}_{k}+{b}_{k},\quad \hat{x}_1=x_{1},\\
{a}_{k+1} &={a}_{k}^r, \quad {a}_1 = a,\\
{b}_{k+1} &={b}_{k}\dfrac{1-{a}_{k+1}}{1-{a}_{k}},\quad {b}_1=b,
\end{align}
\end{subequations}
 satisfying
\begin{equation*}
%{\color{blue}x_k-x^*= a_k(x_{k-1} -x^*)},
%&{\color{blue}\hat{x}_k-x^*=
%\hat{a}_k (\hat{x}_{k-1} - x^*)},\\
%%
x_k-x^*=a^{k-1}(x_1-x^*), \quad
\hat{x}_k-x^*
=a^{r^{k}-r}(\hat{x}_1-x^*),
\end{equation*}
 and
\begin{align*}
\lp\sqrt[k]{|x_k-x^*|}=\lp\sqrt[r^k]{|\hat{x}_k-x^*|}=|a|,
\end{align*}
that is, though $x_k$  converges R-linearly to $x^*$, $\hat{x}_k$ given by~\eqref{doubit}  converges R-superlinearly to $x^*$ with order $r$. Moreover, when $\hat{x}_1\neq x^*$ we have $\hat{x}_k\neq x^*$ for each $k$ and
{
\begin{equation*}
\lp%\sqrt[k]
{\left | \dfrac{\hat{x}_{k+1} - x^*}{(\hat{x}_{k} - x^*)^r} \right |}
%
%= \left | \dfrac{a^{r^2-r}}{(\hat{x}_1-x^*)^{r-1}} \right |
%
=  \dfrac{|a|^{r^2-r}}{|\hat{x}_1-x^*|^{r-1}} .
\end{equation*}}
This shows that $\hat{x}_k$ even converges $Q$-superlinearly to $x^*$ with order $r$. %Note that $\hat{x}_k\neq x^*$ if $\hat{x}_1\neq x^*$ for each $k$.
%
%\begin{align*}
%|f_k-f_\ast|=|a||f_{k-1}-f_\ast|,\quad\mbox{and}\quad |g_k-f_\ast|=\dfrac{|a|}{|f_1-f_\ast|}||g_{k-1}-f_\ast|^2.
%\end{align*}
%Summary, $g_k$ can be designed as the following recursive algorithm,
%\begin{equation*}%\label{doubit}
%\left\{
%\begin{array}{ll}
%g_k &=a_{k}g_{k-1}+b_{k},\quad g_1=f_1\\
%a_k &=a_{k-1}^2,\quad a_1=a.\\
%b_k &=b_{k-1}(1+a_{k-1}),\quad b_1=b.
%\end{array}
%\right.
%\end{equation*}
\end{example}

Note that Example~\ref{simit} provides an accelerated method by constructing the sequence $\{\hat{x}_k\}$, but, in general,  the sequence $\{\hat{x}_k\}$ does not exist or could not be built up as above. In the subsequent discussion, we would like to show that under a certain condition, the fixed-point iteration can be accelerated with the rate of convergence of any desired order.

This work is organized as follows.
%In Section 2, we introduce two nonlinear matrix equations and apply the fixed-point iteration to compute the exact solutions. Moreover,
%the R-superlinearly convergent iterative methods with order $r>1$ are designed to accelerate the convergence speed.
In Section 2, we introduce two nonlinear matrix equations and provide the fixed-point iteration to compute the exact solutions. To accelerate the convergence speed, we propose methods having sequences which are R-superlinearly convergent with order $r$.
In Section 3, we introduce the definition of the semigroup property. Inspired by the examples in Section~2, we show that once an iteration has the semigroup property, the iteration will converge to the solution with the rate of convergence of any desired order.
In Section~4,  we show the practicability and effectiveness of our accelerated iteration by several real-life examples, and concluding remarks are given in Section 5.
%Inspired by the
%
%We show that
%Inspired with the idea given in Section~2, we introduce the definition of the semigroup property for the iteration, $X_{k+1}=F(X_k,X_1)$. Properties of the original iteration and its relationship with the new accelerated iteration are discussed in Section~3. Several real-life examples show that the practicability and effectiveness of our accelerated iteration are given in Section~4. Also, concluding remarks are given in Section 5.

\section{Dynamical behaviors for solving nonlinear scalar equations}\label{sec2}
In this section, we intend to review computational techniques for solving scalar equations and hand accelerated techniques through the observation of the iterative behaviors. Though these two examples are simple and could be solved theoretically, we see how the iterative way can be applied and accelerated to solve these two examples.

% Iterative properties discovered in these two examples will be connected and concluded to discover a rule to build up accelerated methods to solve general matrix equations with the rate of convergence of any desired order later on.

%The first one below, though simple, gives a direct interpretation on how the acceleration occurs.
\begin{example}\label{ex1}
Consider a nonlinear scalar equation:
  \begin{align}\label{eq1}
          x=\dfrac{bx}{b+x-a},
  \end{align}
where $a$ and $b$ are two given different complex numbers and $b\neq 0$.
It is clear that solutions of Eq.~\eqref{eq1} include $x=0$ and $x=a$. To iteratively solve~\eqref{eq1}, we observe the following iteration
\begin{equation}\label{eq:xk0}
x_{k+1}=g(x_{k},b):=\dfrac{bx_{k}}{b+x_k-a}, \quad k = 1, 2, \ldots,
\end{equation}
with $x_1 = b$.
Let $\mathbf{s} = \cup_{p\geq 1}\{e^{\frac{2ji\pi}{p}};0\leq j \leq p-1\}$.
%
%can approach the solution of Eq.~\eqref{eq1}.
Note that
%
%if no breakdown occurs,
the sequence $\{x_k\}$ is well-defined if and only if $ \dfrac{x_1}{x_1-a}\not\in \mathbf{s}$ or $ \dfrac{x_1}{x_1-a}=1$ \rm{(} i.e., $a = 0$\rm{)}, and
%if
%$x_1\neq a$, then it can be shown by induction that for $k\geq 1$,
%\begin{equation}\label{eq:xk}
%{x}_{k+1}-a=\dfrac{(b-a)(x_k-a)}{b+x_k-a}\neq 0,
%\end{equation}
%%that is,
%%$x_{k+1} \neq a$,
%and
the explicit form of $x_k$ can be expressed equivalently as
%  \[
%  x_k=\dfrac{ax_1^k}{x_1^k-(x_1-a)^k}
%  \]
% if $a\neq 0$ and ${x}_{k}=\dfrac{x_1}{k}$ if $a=0$.
\begin{align}
     x_k&=\dfrac{ax_1^k}{x_1^k-(x_1-a)^k},\, { \mbox{if } a\neq 0 \mbox{ and }
      \dfrac{x_1}{x_1-a} \not\in \mathbf{s}}, \\
    x_k&=\dfrac{x_1}{k},\, \mbox{if }  a=0.
\end{align}
This implies that if $\dfrac{x_1}{x_1-a} \not\in \mathbf{s}$, then
%there exists two constants $\rho_1, \rho_2 \in(0,1)$ such that
\begin{align*}
  \limsup\limits_{k\rightarrow\infty}\sqrt[k]{|x_k-0|}&\leq \left |\frac{x_1}{x_1 -a} \right |, \quad\mbox{if }|\frac{x_1}{x_1 -a}|<1, \\
   \limsup\limits_{k\rightarrow\infty}\sqrt[k]{|x_k-a|}&\leq \left|\frac{x_1-a}{x_1} \right |, \quad\mbox{if }|\frac{x_1}{x_1 -a}|>1,
 \end{align*}
%
%
%
% for a sufficiently large $k$,
% \begin{align*}
%  {|x_k- 0 |}&= \left |  \frac{a}{1- (\frac{x_1 -a}{x_1})^k} \right  |
%  %
%  = \left |  \frac{a (\frac{x_1}{x_1 -a})^k
%  }{(\frac{x_1}{x_1 -a})^k- 1} \right  |
%  %
%  \leq |a| \left|\frac{x_1}{x_1 -a}\right|^k, \quad\mbox{if } \left |
%  \frac{x_1}{x_1 -a}
% \right |<1, \\
%  {|x_k- a |}&= \left|  \frac{a}{(\frac{x_1}{x_1 -a})^k -1}\right |
%   %
%  = \left |  \frac{a (\frac{x_1 -a}{x_1})^k
%  }{1- (\frac{x_1 -a}{x_1})^k} \right  |
%  %
%  \leq |a| \left | \frac{x_1 -a}{x_1}\right |^k, \quad\mbox{if } \left |
%  \frac{x_1}{x_1 -a} \right
%  |>1,
% \end{align*}
i.e., we have $x_k$ converges {\rm R}-linearly to $0$ if $\left |
  \frac{b}{b -a}
 \right |<1$ and to  $a$ if $\left |
  \frac{b}{b -a}  \right | >1$,  and  $x_k$ converges {\rm Q}-sublinearly to $0$  if $a = 0$ or $\left |
  \frac{b}{b -a}  \right | =1$.

For any positive integer $r>1$, let $g_r$ be a function defined  recursively by
  \begin{align*}
  h_{r}(x)&:=g(h_{r-1}(x), x)
   \end{align*}
with $h_{1}(x) = x$. Clearly, we have $x_r = h_r(x_1)$ if $x_1 = b$, and
   \begin{align*}
     h_r(x)&=\dfrac{ax^r}{x^r-(x-a)^r},\quad a\neq 0 , \\
    h_r(x)&=\dfrac{x}{r},\quad a=0.
     \end{align*}
To accelerate the iterations induced by the sequence $\{x_k\}$, let $\{y_k\}$ be a new sequence defined by
\begin{equation*}
y_k=h_{r}(y_{k-1})
\end{equation*}
with $y_{1}=x_1$.
Observe that when $a\neq 0$,
\begin{eqnarray*}
1+\frac{a}{y_k-a}
&=& 1 +  \frac{y_{k-1}^r - (y_{k-1}-a)^r}{(y_{k-1}-a)^r}
= \frac{(y_{k-1})^r}{(y_{k-1}-a)^r}
=(1+\frac{a}{y_{k-1}-a})^r
\\
&=&(1+\frac{a}{y_{1}-a})^{r^{k-1}},
\end{eqnarray*}
or, %Let $\hat{y}_{1}:=\frac{y_1}{y_1-a}$. Then
\begin{equation*}
 y_k - a = \frac{a}{ (\frac{y_1}{y_1-a})^{r^{k-1}} - 1}.
\end{equation*}
Thus, we have
\begin{align*}
  \limsup\limits_{k\rightarrow\infty}\sqrt[r^k]{|y_k-0|}&\leq  \left  |\frac{y_1}{y_1-a} \right  |, \quad\mbox{if } \left  | \frac{y_1}{y_1-a} \right  |<1, \\
   \limsup\limits_{k\rightarrow\infty}\sqrt[r^k]{|y_k-a|}&\leq
   \left  |\frac{y_1-a}{y_1} \right  |, \quad\mbox{if } \left  |\frac{y_1}{y_1-a} \right  | >1.
 \end{align*}
%
%
%\begin{align*}
%  {|y_k- 0 |}&= \left |  \frac{a
%(\frac{y_1}{y_1-a})^{r^{k-1}} }
%  {(\frac{y_1}{y_1-a})^{r^{k-1}} -1} \right  |
%  %
%  %
%  \leq |a|
%  \left|\frac{y_1}{y_1 -a}\right|
%  \left|\frac{y_1}{y_1 -a}\right|^{r^{k}}, \quad\mbox{if } \left |
%  \frac{y_1}{y_1 -a}
% \right |<1, \\
%    {|y_k- a |}&= \left |  \frac{a
%(\frac{y_1}{y_1-a})^{-r^{k-1}} }
%  {1- (\frac{y_1}{y_1-a})^{-r^{k-1}} } \right  |
%  %
%  %
%  \leq |a|
%  \left|\frac{y_1}{y_1 -a}\right|
%  \left|\frac{y_1}{y_1 -a}\right|^{-r^{k}}, \quad\mbox{if } \left |
%  \frac{y_1}{y_1 -a}
% \right |>1.
% \end{align*}
%
Moreover, when $a =0$ or $\left |
  \frac{b}{b -a}  \right | =1$,  ${y}_{k+1}=\dfrac{y_k}{r}$, i.e., $y_k$ converges to $0$ {\rm Q}-linearly with rate $\dfrac{1}{r}$.
  \end{example}

In Example~\ref{ex1}, we point out that once the parameter $b$ is viewed as a variable, the iteration~\eqref{eq:xk0}, applied to~\eqref{eq1}  with $x_1 = b$, provides a faster convergence.  We thus are interested in investigating whether such observation is not a special case and is available for a general problem. One immediate extension can be seen as below.
\begin{example}\label{ex2}
Let $a$ and $b$ be two complex numbers such that $ab\neq 0$.
Consider the nonlinear equation
  \begin{align}\label{eq2}
\bb x\\ y\eb=\bb \dfrac{a}{a+y}x \\ \dfrac{b}{a+y}y  \eb.
  \end{align}
As for~\eqref{eq2}, there are infinite solutions and can be collected by the set $S=[0,b-a]^\top \cup \left \{[x,0]^\top|x\in\mathbb{C}\right \}$. To iteratively solve~\eqref{eq2}, we consider the iteration
\begin{equation*}
Z_{k+1}=G(Z_k,Z_1):= \bb \dfrac{x_1}{x_1+y_k}x_k \\ \dfrac{y_1}{x_1+y_k}y_k  \eb,
\end{equation*}
where
%and
  $Z_k = \bb x_k \\y_k\eb$, for $k\geq 1$, $x_1 = a$, and $y_1 = b$,
  if no breakdown occurs. When $
{\dfrac{x_1}{y_1} \not\in \mathbf{s}}$, it can be verified that the iteration $Z_k$ can be continued and
\begin{align*}
  x_{k}&=\dfrac{x_1-y_1}{x_1^k-y_1^k}x_1^k,\\
  y_{k}&=\dfrac{x_1-y_1}{x_1^k-y_1^k}y_1^k,
  \end{align*}
   and $x_k=y_k=\dfrac{x_1}{k}$ if $x_1=y_1$. Furthermore, the convergence behavior of $Z_k$ includes three cases:
 \begin{align*}
\begin{array}{l}
{\rm (1)} \mbox{ if }  \left  |\frac{y_1}{x_1} \right | < 1, \mbox{ then}
\left \{
\begin{array}{l}
 \limsup\limits_{k\rightarrow\infty}\sqrt[k]{|x_k-(x_1-y_1)|} \left  |\frac{y_1}{x_1} \right |,\\
  \limsup\limits_{k\rightarrow\infty}\sqrt[k]{|y_k - 0|}\leq  \left  |\frac{y_1}{x_1} \right |,
  \end{array}
\right .  \\
{\rm(2)} \mbox{ if }  {x_1}= {y_1}, \mbox{ then}
\left \{
\begin{array}{l}
 \limsup\limits_{k\rightarrow\infty}\frac{|x_{k+1}-0|}{|x_k-0|} = 1,\\
  \limsup\limits_{k\rightarrow\infty}\frac{|y_{k+1}-0|}{|y_k-0|}= 1,
  \end{array}
\right .  \\
{\rm (3)} \mbox{ if }  \left  |\frac{y_1}{x_1} \right | > 1, \mbox{ then}
\left \{
\begin{array}{l}
\limsup\limits_{k\rightarrow\infty}\sqrt[k]{|{x}_k-0|}
\leq  \left  |\frac{y_1}{x_1} \right |,\\
\limsup\limits_{k\rightarrow\infty}\sqrt[k]{|{y}_k-({y}_1-{x}_1)|}
\leq  \left  |\frac{y_1}{x_1} \right |,
  \end{array}
\right .
\end{array}
\end{align*}
  i.e.,
  %   we have
 \begin{align*}
  &{\rm (1)}  \mbox{ if } \left  |\frac{y_1}{x_1} \right | < 1, \mbox{ then } Z_k\rightarrow  \bb x_1-y_1\\0\eb\mbox{R-linearly},\\
  &{\rm (2)} \mbox{ if }  {x_1} = {y_1} , \mbox{ then } Z_k\rightarrow  \bb 0\\0\eb\mbox{Q-sublinearly},\\
  &{\rm (3)}
  \mbox{ if }   \left  |\frac{y_1}{x_1} \right | > 1, \mbox{ then }  Z_k\rightarrow  \bb 0\\y_1-x_1\eb\mbox{R-linearly}.
  \end{align*}
To speed up the convergence, let $H_r$, $r>1$, be an iterated function defined recursively by
  \begin{align}\label{eq:hrz}
  H_{r}(Z)&:=G(H_{r-1}(Z),Z),\\
  H_{1}(Z)&:=Z,
  \end{align}
  where $Z = \bb x\\y \eb$. Now, repeatedly applying~\eqref{eq:hrz} generates
 \[
H_r(Z)=\frac{1}{\sum\limits_{j=0}^{r-1} x^j y^{r-1-j}}\bb x^r \\ y^r \eb.
\]
Setting
\[
\widehat{Z}_k=\bb \hat{x}_k\\ \hat{y}_k \eb:= H_r ( \widehat{Z}_{k-1}),
\]
with $\bb \hat{x}_1\\ \hat{y}_1 \eb=\bb {x_1}\\ {y_1} \eb$, we can prove by induction that
\begin{align*}
\hat{x}_k & =\dfrac{x_1-y_1}{x_1^{r^{k-1}}-y_1^{r^{k-1}}}{x}_1^{r^{k-1}},\\
\hat{y}_k & = \dfrac{x_1-y_1}{x_1^{r^{k-1}}-y_1^{r^{k-1}}}{y}_1^{r^{k-1}},
\end{align*}
if {$\dfrac{x_1}{y_1} \not\in \mathbf{s}$, and $\hat{x}_k=\hat{y}_k=\dfrac{\hat{x}_{k-1}}{r}$ if $x_1=y_1$. Overall,} this leads to interesting convergence properties, i.e.,
\begin{align*}
  &{\rm (1)}  \mbox{ if } \left  |\frac{y_1}{x_1} \right | < 1, \mbox{ then } \widehat{Z}_k\rightarrow  \bb x_1-y_1\\0\eb\mbox{R-superlinearly with order $r$},\\
  &{\rm (2)} \mbox{ if }  {x_1} = {y_1} , \mbox{ then } \widehat{Z}_k\rightarrow  \bb 0\\0\eb\mbox{Q-linearly},\\
  &{\rm (3)}
  \mbox{ if }   \left  |\frac{y_1}{x_1} \right | > 1, \mbox{ then }  \widehat{Z}_k\rightarrow  \bb 0\\y_1-x_1\eb\mbox{R-superlinearly with order $r$}.
  \end{align*}

%
% \begin{align*}
%  &\widehat{Z}_k\rightarrow  \bb 0\\y_1-x_1\eb\mbox{R-superlinearly with order $r$}\,\,{\mbox if }\,\, \left|\frac{y_1}{x_1}\right |>1,\\
%  &\widehat{Z}_k\rightarrow  \bb 0\\0\eb\mbox{R-superlinearly}\,\,\mbox{if }\,\, \frac{y_1}{x_1} = 1,\\
%  &\widehat{Z}_k\rightarrow  \bb x_1-y_1\\0\eb\mbox{R-superlinearly with order $r$}\,\,\mbox{if }\,\, \left|\frac{y_1}{x_1}\right |<1.
%  \end{align*}
%

  \end{example}

Regarding the accelerated methods discussed in Example~\ref{ex1} and Example~\ref{ex2}, we would like to ask whether there is a general rule to follow, i.e., a connection between these two examples. As a consequence, we observe that in Example~\ref{ex1}, for $x,y,z\in\mathbb{C}$, the operation $g$ satisfies
\begin{subequations}~\label{eq:ob1}
\begin{align}
g(g(x,y),z)&=\dfrac{z}{z+g(x,y)-a}g(x,y)=\dfrac{z}{z+\frac{y}{y+x-a}x-a}(\frac{y}{y+x-a}x)\\
&=\dfrac{xyz}{xy+(z-a)(x+y-a)}=\dfrac{xyz}{zy+(x-a)(z+y-a)}\\
&=\dfrac{\frac{z}{z+y-a}y}{\frac{z}{z+y-a}y+x-a}x=\dfrac{g(y,z)}{g(y,z)+x-a}x=g(x,g(y,z)).
\end{align}
\end{subequations}
Also, for
$X=\bb x_1\\x_2\eb,Y=\bb y_1\\y_2\eb,Z=\bb z_1\\z_2\eb\in\mathbb{C}^{2\times 1}$, the operation $G$ in Example~\ref{ex2} satisfies
%
%\color{red} Something wrong? $F(X,Y)$ is a 2-by1 vector, not a scalar.
\begin{subequations}~\label{eq:ob2}
\begin{align}
G(G(X,Y),Z)&=\dfrac{1}{z_1+\dfrac{y_2x_2}{y_1+x_2}}\bb \dfrac{y_1x_1}{y_1+x_2}z_1\\ \dfrac{y_2x_2}{y_1+x_2}z_2 \eb
=\dfrac{1}{z_1y_1+z_1x_2+y_2x_2}\bb x_1y_1z_1\\ x_2y_2z_2 \eb\\
&=\dfrac{1}{x_2+\dfrac{z_1y_1}{z_1+y_2}}\bb \dfrac{z_1y_1}{z_1+y_2}x_1\\ \dfrac{z_2y_2}{z_1+y_2}x_2 \eb=G(X,G(Y,Z)).
\end{align}
\end{subequations}

From~\eqref{eq:ob1} and~\eqref{eq:ob2}, we see that the matrix operators, $g$ and $G$, defined in Examples~\ref{ex1} and~\ref{ex2}, satisfy a specific property, the associative law. In the coming section, we would like to show that for this property, we can come up with an accelerated way for the original iteration.

\section{The semigroup action on a binary operator}
In the introduction, we express a way to speed up the fixed-point iteration, and even more, in Section~\ref{sec2}, we show that both accelerated methods have a common feature, the associative rule, observed by the definition of matrix operators. Now, we are interested to see that if the matrix operator associated with the fixed-point iteration has the associative rule, the accelerated method is always attainable. Our concept is started from the observation of the ``flow" in the ordinary differential equations.   Note that if $\hat{x}(t;x_0) = x(t)$ is a solution of the autonomous system with an initial value $x_0$ and $f\in\mathbb{C}^1$, i.e.,
%\begin{subequations*}\label{ivp}
\begin{align*}
 \dot{x}&=f(x(t)),\, t\geq0,\\
 x(0)&=x_0{\mbox{ is given}}.
\end{align*}
%\end{subequations*}
Then the solution $\hat{x}$ must satisfy the following \emph{group property}
 \[
 \hat{x}(s+t;x_0)=\hat{x}(s; \hat{x}(t;x_0))
 =\hat{x}(t; \hat{x}(s;x_0)), \quad s,\,t\geq0.
 \]
For a discrete iteration, consider a difference equation defined by
%
 %\begin{subequations}
\begin{align}\label{ivp1}
 x_{k+1}&=F(x_k),\quad k\geq1,\\
 x_1&{\mbox{ is given}}, \nonumber
\end{align}
%\end{subequations}
where $F$ is an iterative operator.
Let $\hat{x}_{k;x_1} = x_k$ be a solution of~\eqref{ivp1}.
Like the autonomous system, the value of $\hat{x}$ in~\eqref{ivp1} is related to the selection of the initial value.  That is, the intermediate results of the action of the operator $F$ are affected by the initial value $x_1$. Without loss of generality, we can regard $F$ as a binary matrix operator and consider the iteration in~\eqref{ivp1} in terms of
\begin{equation}\label{eqF2}
x_{k+1} = F(x_k,x_1),
\end{equation}
where $x_1$ is thought of as another variable. Our subsequent discussion is to see that under what circumstance the following associative event
 \begin{align}\label{group0}
\hat{x}_{s+t; x_1}
 =\hat{x}_{s; \hat{x}_{t;x_1}}
 =\hat{x}_{t; \hat{x}_{s;x_1}},\quad s,t\geq 1,
 \end{align}
 is satisfied if $F$ is an operator of two variables.
% We then apply this event seen in~\eqref{group0} to develop an accelerated technique for the original fixed-point iteration method.
In Example~\ref{simit}, we use this property~\eqref{group0} to accelerate the original fixed-point iteration. In our subsequent discussion, we see how to expedite the fixed-point iteration in~\eqref{eqF2} once this feature seen in~\eqref{group0} is satisfied.

To start with, we have to introduce the definition of the semigroup property.  We show that once the iteration,  as is given in~\eqref{eqF2}, has a semigroup property, we can build up an accelerated iteration with the rate of convergence of any desired order. In the algebra, we say that a set has the semigroup property if this set together with an operator must satisfy an associative rule.  To our purpose, we extend this definition to fit for the discrete iterations.

\begin{Definition}\label{def1}
Let $D\subseteq \mathbb{C}^{n\times m}$ and $F:D\times D\rightarrow D$ be a binary matrix operator. We call that an iteration
\begin{equation}\label{eqF3}
X_{k+1}=F(X_k,X_1), \quad k\geq 1,
\end{equation}
has the
 semigroup property if the operator $F$ satisfies the following associative rule:
\[
F(F(X,Y),Z)=F(X,F(Y,Z)),
\]
for any $X,Y$ and $Z$ in $D$.
%That is, we call the binary operator $F$ an association binary operator.
\end{Definition}

With Definition~\ref{def1}, we claim step by step how to set up an accelerated technique by applying the following Sherman MorrisonWoodbury formula (SMWF) in our proofs.

\begin{Theorem}\cite{Bernstein2005}:\label{Schur}
Let $A$, $B$, $U$, and $V$ be matrices of size $n$. If $U$, $V$, and $V^{-1}\pm AU^{-1}B$ are nonsingular, then
$U\pm B V A$ is invertible and
\[
(U\pm B V A)^{-1}=U^{-1}\mp U^{-1}B(V^{-1}\pm AU^{-1}B)^{-1}AU^{-1}.
\]

\end{Theorem}

Before, we go further to the investigation of accelerated techniques. We consider one type of binary operator $F$,
%, i.e.,
%\begin{equation}
%F(X,Y)=X\Delta_{X,Y}Y,
%\end{equation}
which can be viewed as a matrix representation of Example~\ref{simit}, to demonstrate the existence of an operator which satisfies the requirement given in Definition~\ref{def1}.

\begin{example}
Let $A$ be an arbitrary matrix with size $n\times n$.
For any three $n$-square matrices $X$, $Y$ and $Z$. Let $\Delta_{X,Y}=(A+X+Y)^{-1}$ if $A+X+Y$ is nonsingular and let $F$ be a binary matrix function defined by
\[F(X,Y)=X\Delta_{X,Y}Y.\]
With the aid of the {Sherman Morrison Woodbury formula}, it can be shown that
\begin{align*}
 \Delta_{X,F(Y,Z)}%&=(A+X+Y-Y\Delta_{Y,Z}(A+Y))^{-1}
 &=\Delta_{X,Y}+\Delta_{X,Y}Y\Delta_{F(X,Y),Z}(A+Y)\Delta_{X,Y},\\
 \Delta_{F(X,Y),Z}&=\Delta_{Y,Z}+\Delta_{Y,Z}(A+Y)\Delta_{X,F(Y,Z)}Y\Delta_{Y,Z},
\end{align*}
if either $\Delta_{X,F(Y,Z)}$ or $\Delta_{F(X,Y),Z}$ exists and $A+X+Y$ and $A+Y+Z$ are both nonsingular. Thus,
\begin{align*}
&(A+X+Y)\Delta_{X,F(Y,Z)}Y=(\Delta_{X,Y}^{-1}\Delta_{X,F(Y,Z)})Y\\
&=(I+Y\Delta_{F(X,Y),Z}(A+Y)\Delta_{X,Y})Y=Y\Delta_{F(X,Y),Z}(A+F(X,Y)+(A+Y)\Delta_{X,Y}Y+Z)\\
&=Y\Delta_{F(X,Y),Z}(A+X\Delta_{X,Y}Y+(A+Y)\Delta_{X,Y}Y+Z)=Y\Delta_{F(X,Y),Z}(A+Y+Z),
\end{align*}
or $\Delta_{X,F(Y,Z)}Y\Delta_{Y,Z}=\Delta_{X,Y}Y\Delta_{X,F(Y,Z)}$.
Then, the iteration $X_{k+1}=F(X_k,X_1)$ has the semigroup property since
\begin{align*}
F(F(X,Y),Z)= X\Delta_{X,Y}Y\Delta_{F(X,Y),Z}Z=X\Delta_{X,F(Y,Z)}Y\Delta_{Y,Z}Z=F(X,F(Y,Z)),
\end{align*}
for appropriate matrices $X$, $Y$, and $Z$.

\end{example}

Next,  we prove in general that if a sequence given by an operator has the semigroup property, the iterations of the sequence must follow a ``discrete flow".

\begin{Theorem}\label{thm:flow}
Let $\{X_k\}$ in~\eqref{eqF3} be a sequence with the semigroup property. Then the sequence $\{X_k\}$ satisfies the following ``discrete flow property'':
\begin{align}\label{eqflow}
X_{i+j}=F(X_i,X_j),\quad i, j \geq 1.
\end{align}
\end{Theorem}
\begin{proof}
Upon using the principle of mathematical induction, we divide our proof in two steps. First, when $j=1$, we show that $X_{i+1}=F(X_1,X_i)$.  It is clear for $i = 1$, the formula~\eqref{eqflow} holds. Assume that the formula is true for $i=s$. Then,
      \begin{align*}
        X_{s+2} =F(X_{s+1},X_1)=F(F(X_{s},X_1),X_1)=F(F(X_{1},X_{s}),X_1)=F(X_1,X_{s+1}),
      \end{align*}
which completes the proof of the first part.
Second, assume that $X_{i+j}=F(X_i,X_j)$ holds for $j=s$. Then,
\begin{align*}
  X_{i+s+1}&=F(X_{i+1},X_s)=F(F(X_{i},X_1),X_s)=F(X_i,X_{s+1}).
\end{align*}
\end{proof}

Note that Theorem~\ref{thm:flow} provides an intrinsic characterization of an iteration with semigroup property. We now illustrate how the satisfaction of the semigroup property can give rise to a new accelerated iteration based on the original iteration. Below we assume that $F$ is an operator, which provides with the situation that iteration~\eqref{eqF3} has the semigroup property. Based on Theorem~\ref{thm:flow}, we propose the following algorithm for computing the iteration~\eqref{eqF3} and show that the speed of convergence can be any desired order later.

\begin{Algorithm}\label{aa2}
{\emph{(The accelerated algorithm for computing $X_{k+1}=F(X_k,X_1)$)}}
\begin{enumerate}

\item {Given a positive integer $r>1$, let $\widehat{X}_1=X_1$;}

\item {For} $k= 1,\ldots,$ iterate
 \begin{align*}
\widehat{X}_{k+1}& =F(\widehat{X}_k,X_{k}^{(r-1)}),
\end{align*}
    until convergence, where $X_{k}^{(r-1)}$ is defined in step 3.
\item
     {For} $\ell=1,\ldots,r-2$, iterate
   \begin{align*}
X_{k}^{(\ell+1)}& =F(\widehat{X}_k,X_{k}^{(\ell)}),
\end{align*}
with $X_{k}^{(1)}=\widehat{X}_k$.
 \end{enumerate}
\end{Algorithm}

For clarity, we list in Table~\ref{Tab1} the results of the iterations obtained by Algorithm~\ref{aa2}.
\begin{table}[h!!!]
  %\caption{}
  \centering
  \begin{tabular}{c|cccc}\hline
    $k$ & 1 &  2 & $\cdots$ &$m$ \\
    \hline \rule{0pt}{2.3ex}
  $ \widehat{X}_k$ & $X_1$ &  $X_r$& $\cdots$ & $X_{r^{m-1}}$ \\
%     $X_k^{(2)}$ & $X_2$ &  $X_{2r}$& $\cdots$ & $X_{2 r^{m-1}}$ \\
  $X_k^{(\ell)}$ & $X_\ell$ &  $X_{\ell r}$& $\cdots$ & $X_{\ell r^{m-1}}$ \\
% $X_k^{(r-1)}$ & $X_{r-1}$ &  $X_{(r-1)r}$& $\cdots$ & $X_{(r-1)r^{m-1}}$ \\
 %\\
% $\widehat{X}_{k+1}$ & $X_r$ &  $X_{r^2}$& $\cdots$ & $X_{r^{m}}$ \\
        \hline
  \end{tabular}
      \caption{Records of iterations by Algorithm~\ref{aa2} with
  $1\leq \ell \leq r-1$ and $k = 1,\ldots, m$.}
\label{Tab1}
\end{table}
%
%
%
%
%From Table~\ref{Tab1}, w
This implies that if the sequence $\{{X}_k\}$ converges R-linearly to $X_\ast$ and $\|X_k-X^\ast\|=o(\rho^k)$ for a constant $0<\rho<1$, we have $\|\widehat{X}_k-X_\ast\|=o(\rho^{r^k})$. In other words, the
sequence $\{\widehat{X}_k\}$ will converges R-superlinearly to $X_\ast$ with order $r$ provided that the sequence $\{{X}_k\}$ converges R-linearly to $X_\ast$.

Specifically, We see that if $r=2$, Algorithm~\ref{aa2} reduces to the so-called ``doubling algorithm''
\begin{align*}
 \widehat{X}_{k+1}=F(\widehat{X}_k,\widehat{X}_k),
\end{align*}
%with starting value $\widehat{X}_1=X_1$,
and  $\widehat{X}_{k}$ converges to $X^\ast$ R-quadratically. Similarly, once $r = 3$, Algorithm~\ref{aa2} reduces to the so-called ``tripling algorithm''
\begin{align*}
 \widehat{X}_{k+1}=F(\widehat{X}_k,F(\widehat{X}_k,\widehat{X}_k))=F(F(\widehat{X}_k,\widehat{X}_k),\widehat{X}_k),
\end{align*}
and
$\widehat{X}_{k}$ converge to $X_\ast$ R-cubically.
%
%Particularly,

%
%
%{\color{red}
%
%
%\begin{Algorithm}\label{aa1}
%{\emph{(The accelerated algorithm for computing for $X_{k+1}=F(X_k,X_1)$)}}
%\begin{enumerate}
%
%\item {Given a positive integer $\ell>1$,let $Y_1=X_\ell$;}
%
%\item  {For} $k= 1,\ldots,$ iterate
% \begin{align*}
%Y_{k+1}& =F(Y_k,Y_1),
%\end{align*}
%    until convergence.
% \end{enumerate}
%\end{Algorithm}
%Since $F$ satisfies the semigroup property, we conclude that
% $Y_{k}=X_{\ell k}$.
%
%
%}

%%%%
%
%
%  Numerical applications
%
%
%%%%
\section{Real world problems}
We must emphasize that the applications of our way depend on the notion of what people want and don't have it. For us, this work is to discover a characteristic to speed up the standard iterative ways. The possible real-life examples and numerical experiments are referred to the ones given in~\cite{Smith1968, Zhou2009, Chiang2014, Lin2015, Chiang2016, Chiang2017, Lin2018a, Lin2018b}, while a much heavier demonstration is carried out. Unlike the existing results, we demonstrate these examples concerning the insight of the semigroup property.  We show that once the semigroup property is satisfied, our accelerated technique is available immediately by the employment of Algorithm~\ref{aa2}.
%
%We must emphasize that the applications of our way depend on the notion of what people want and don't have it. For us, this work is to discover a characteristic to speed up the standard iterative ways. The possible real-life examples and numerical experiments are referred to the ones given in~\cite{Smith1968, Zhou2009, Chiang2014, Lin2015, Chiang2016, Chiang2017, Lin2018a, Lin2018b}, while a much heavier demonstration is carried out. In this section, we recall these examples and show that the satisfaction of the semigroup property in each example so that our accelerated technique given by Algorithm~\ref{aa2} is available immediately.
%
%We should clarify that there exists an accelerated iteration in each example.
%These methods can be obtained by using the approach of Section~3, once the semigroup property can be verified. Furthermore, we provided that an insight into the essential
%of the original iteration, unlike the existing approach.
%
%
% Ex2:The Stein matrix equation
%
%
%\subsection{Ex2:The Stein matrix equation}

\begin{example}
Consider the Stein matrix equation
\begin{align}\label{Stein}
X=F(X):=AXB+C,
\end{align}
where $A\in\mathbb{C}^{m\times m}$, $B\in\mathbb{C}^{n\times n}$, $C\in\mathbb{C}^{m\times n}$ are known matrices and $X\in \mathbb{C}^{m\times n}$  is a matrix to be determined.
Eq.~\eqref{Stein} represents a model commonly encountered in the applications of control theory~\cite{Lancaster95} and is well-known to the numerical linear algebra community.
%The linear matrix equation~\eqref{Stein} is famous for  numerical linear algebra and is encountered in some applications of control theory~\cite{Lancaster95}.
%
Suppose that $\rho(A)\rho(B)<1$, then the unique solution
\[
X_* = \sum_{i=0}^\infty A^i C B^i
\]
 to~\eqref{Stein} exists. While applying the fixed-point iteration to~\eqref{Stein}, we see that
\[
X=F^{(k)}(F(X))=F^{(k+1)}(X):=A_{k+1} X B_{k+1}+C_{k+1},
\]
or
\begin{subequations}\label{Smith}
\begin{align}
A_{k+1}&=A_k A_1,\\
B_{k+1}&=B_1 B_k,\\
C_{k+1}&=C_k+A_k C_1 B_k,
\end{align}
\end{subequations}
with the initial value $(A_1,B_1,C_1)=(A,B,C)$. Note that the iteration~\eqref{Smith} is so-called the Smith iteration~\cite{Smith1968} and the speed of convergence can be denoted by
\[
\limsup\limits_{k\rightarrow\infty}\sqrt[k]{\|C_k-X_*\|}\leq \rho(A)\rho(B).
\]

Like~\eqref{eqF2}, we extend the domain of $F$ and consider the mapping $F$ as an action defined by
 \begin{align}\label{type1}
  F(X_{a},X_{b}):=\bb A_{a}A_{b} \\  B _{b}B_a \\
  C_{a}+A_{a}C_{b}B_a\eb,
  \end{align}
where $X_i=[ A_i, B_i, C_{i}] ^\top$,
 for $i = a, b$, lies in the domain of $F$.
Using~\eqref{type1}, the iterations of~\eqref{Smith} can be rewritten as
  \begin{align*}
  X_{k+1}=F(X_{k},X_{1}).
  \end{align*}
Note that the property of~\eqref{eqflow} is satisfied because
  \begin{align*}
&F(F(X_{a},X_{b}),X_{c}) =\bb (A_{a}A_{b})A_{c} \\ B_{c}(B _{b}B_a) \\
 (C_{a}+A_{a}C_{b}B_a)+ (A_aA_b)C_c(B_b B_a)\eb\\
&=\bb A_{a}(A_{b}A_{c}) \\ (B_{c}B _{b})B_a \\
 C_{a}+A_{a}(C_{b}+ A_b C_c B_b) B_a \eb
=F(X_{a},F(X_{b},X_{c})).
\end{align*}
%for any $a,b,c \geq 1$.
%
This gives rise to the fact that using Algorithm~\ref{aa2} to solve~\eqref{Stein} yields the
$r$-Smith iteration~\cite{Zhou2009}:
\begin{subequations}\label{rSmith}
\begin{align}
\widehat{A}_{k+1}&=\widehat{A}_k^r,\\
\widehat{B}_{k+1}&=\widehat{B}_k^r,\\
\widehat{C}_{k+1}&=\sum\limits_{\ell=0}^{r-1}\widehat{A}_k^\ell \widehat{C}_k \widehat{B}_k^\ell.
\end{align}
\end{subequations}
with the initial value $(\widehat{A}_1,\widehat{B}_1,\widehat{C}_1)=(A,B,C)$ and the speed of convergence is $R$-superlinear with order $r$, i.e.,
\[
\limsup\limits_{k\rightarrow\infty}\sqrt[r^k]{\|\widehat{C}_k-X_*\|}\leq \rho(A)\rho(B).
\]

\end{example}

%
%
% Ex3:The generalized eigenvalue problem
%
%

%\subsection{Ex3:The generalized eigenvalue problem}
\begin{example}
 Given a regular $n\times n$ matrix pencil $A- \lambda B$ (i.e., $\det(A-\lambda B)$ is not identically zero for all $\lambda$) and an integer $m\leq n$, we want to find in this example a {full rank matrix} $U\in\mathbb{C}^{n\times m}$ such that
\begin{equation*}
A U = B U  \Lambda,
\end{equation*}
where $\Lambda\in\mathbb{C}^{m\times m}$ and {$\rho(\Lambda) < 1$}. The column space of $U$ is called a stable subspace of the matrix pencil $A- \lambda B$. It is well known that a class of numerical methods for solving some matrix equations can be reduced to the computation of a kind of stable subspace of a suitable matrix pencil~\cite{Bini2012,Lancaster95,Huang2018}.

To this end, let $\Delta_{1,k}:=(A_1+B_{k})^{-1}$ and consider the following iterations,
\begin{subequations}\label{ex3fix}
\begin{align}
  A_{k+1}&=A_1\Delta_{1,k}A_{k}=A_{k}-B_{k}\Delta_{1,k}A_{k},\\
  B_{k+1}&=B_{k}\Delta_{1,k}B_{1}=B_{1}-A_{1}\Delta_{1,k}B_{1},
\end{align}
\end{subequations}
for $k\geq 1$. It has been shown in~\cite{Lin2018b} that
if the sequence defined by~\eqref{ex3fix} exists and $A_1U = B_1U\Lambda$, then  $A_k U= B_k U \Lambda^k$.
%\end{Theorem}
This implies that
 if $\rho (\Lambda)< 1$, and if the sequence $\{B_k\}$  is uniformly bounded, then $\lim\limits_{k\rightarrow \infty} A_k U = 0$. To solve the solution $U$ is equal to compute the right null space of $A_\infty$, where $A_\infty := \lim\limits_{k\rightarrow \infty} A_k$. Our subsequence discussion is to show that how Algorithm~\ref{aa2} can be directly applied to~\eqref{ex3fix} to accelerate the iterations.
 A similar but more complex discussion can be found in~\cite[Section~2]{Lin2018b}.

To start with, let $a$ and $b$ be two dummy indices such that  %$A_a- \lambda B_a$ represents an $n\times n$ regular matrix pencil,
$X_{a}=\bb A_{a} \\ B_{a}\eb$, $\Delta_{a,b} =(A_{a}+B_{b})^{-1}$, and
  \begin{align}\label{type1}
  F(X_{a},X_{b})=\bb A_{a}\Delta_{a,b}A_{b} \\ B_{b}\Delta_{a,b}B_{a}\eb=\bb A_{b}-B_{b}\Delta_{a,b}A_{b}\\B_{a}-A_{a}\Delta_{a,b}B_{a} \eb,
  \end{align}
  are well-defined,
  that is, $X_{k+1}=F(X_{k},X_{1})$ for $k\geq 1$.  Similarly, let $X_{d}=F(X_{a},X_{b})$, $X_{e}=F(X_{b},X_{c})$, $X_{\ell}=F(X_d,X_{c})$ and $X_{r}=F(X_{a},X_{e})$. We see that
\begin{align*}
 \Delta_{d,c}&=(A_{d}+B_{c})^{-1}=(A_{b}+B_{c}-B_{b}{\Delta_{a,b}}A_{b})^{-1}\\
 &=
 {\Delta_{b,c}+\Delta_{b,c}B_{b}\Delta_{a,e}A_{b}\Delta_{b,c},}\\
 \Delta_{a,e}&=(A_{a}+B_{e})^{-1}=(A_{a}+B_{b}-A_{b}\Delta_{b,c}B_{b})^{-1}\\
 &=\Delta_{a,b}+\Delta_{a,b}A_{b}\Delta_{d,c}B_{b}\Delta_{a,b},\\
 \Delta_{a,b}^{-1}\Delta_{a,e}&=(A_{a}+B_{e}+A_{b}\Delta_{b,c} B_{b})(A_{a}+B_{e})^{-1}=I_n+A_{b}\Delta_{b,c} B_{b} \Delta_{a,e},\\
 \Delta_{a,e}\Delta_{a,b}^{-1}&=(A_{a}+B_{e})^{-1}(A_{a}+B_{e}+A_{b}\Delta_{b,c} B_{b})=I_n+\Delta_{a,e}A_{b}\Delta_{b,c}B_{b}.
  \end{align*}
Furthermore,
 \begin{align*}
  \Delta_{a,b} A_{b} \Delta_{d,c}&=\Delta_{a,b}A_{b}\Delta_{b,c}+\Delta_{a,b}A_{b}\Delta_{b,c}B_{b}\Delta_{a,e}A_{b}\Delta_{b,c}\\
  &=\Delta_{a,b}(I_n+A_{b}\Delta_{b,c}B_{b}\Delta_{a,e})A_{b}\Delta_{b,c}\\
  &=\Delta_{a,e} A_{b} \Delta_{b,c},\\
 \Delta_{d,c}B_{b}\Delta_{a,b}&=\Delta_{b,c}B_{b}\Delta_{a,b}+\Delta_{b,c}B_{b}\Delta_{a,e}A_{b}\Delta_{b,c}B_{b}\Delta_{a,b}\\
  &=\Delta_{b,c}B_{b}(I_n+\Delta_{a,e}A_{b}\Delta_{b,c}B_{b})\Delta_{a,b}\\
  &=\Delta_{b,c} B_{b} \Delta_{a,e}.
  \end{align*}
After these preliminaries, we can state that
\begin{align*}%\label{comm}
{F(F(X_{a},X_{b}),X_{c})=F(X_{a},F(X_{b},X_{c}))}
\end{align*}
by the following observation
\begin{align*}
 A_{\ell}&=A_{d} \Delta_{d,c} A_{c}=(A_{a} \Delta_{a,b} A_{b}) \Delta_{d,c}A_{c} =A_{a} (\Delta_{a,b} A_{b} \Delta_{d,c}) A_{c}\\
 &=A_{a} (\Delta_{a,e} A_{b}\Delta_{b,c} ) A_{c}=A_{a} \Delta_{a,e} (A_{b}\Delta_{b,c}  A_{c})=A_{a} \Delta_{a,e} A_{e}=A_r,\\
B_{\ell                                                                                                                                         }&=B_{c} \Delta_{d,c} B_{d}= B_{c}\Delta_{d,c}(B_{b}\Delta_{a,b} B_{a})  =B_{c} (\Delta_{d,c} B_{b} \Delta_{a,b}) B_{a}\\
&=B_{c} (\Delta_{b,c} B_{b} \Delta_{a,e}) B_{a}=(B_{c} \Delta_{b,c} B_{b}) \Delta_{a,e} B_{a}=B_{e} \Delta_{a,e} B_{a}=B_r.
\end{align*}
That is, an accelerated technique is available by applying Algorithm~\ref{aa2}.
More details on accelerated algorithm can be found in~\cite{Lin2018b}.

\end{example}

%
%
%Ex4:The generalized nonlinear matrix equation
%
%
%
\begin{example}
%~\cite{Chiang2017}
In this example, we consider a kind of nonlinear matrix equations~\cite{Bini2012,Huang2018,Chiang2017}
%{
\begin{align}\label{eq:NME}
X=N(X):=Q-AX^{-1}B,
\end{align}
where $A$, $B$, and $Q$ are $n\times n$ matrices.
%}
%be a nonlinear matrix equation.
If $X$ is a solution of~\eqref{eq:NME},
we see that for $k\geq 1$,
\begin{align*}
X= N^{(k)}(N(X)) =N^{(k+1)}(X):=Q_{k+1}-A_{k+1}(X-P_{k+1})^{-1}B_{k+1},
 \end{align*}
with matrices $A_k$, $B_k$, $Q_k$ and $P_k$ defined by
  \begin{subequations}\label{ex4fix}
 \begin{align}
A_{k+1}&:=A_{k}\Delta_{1,k}A_1,\\
B_{k+1}&:=B_1\Delta_{1,k}B_{k},\\
P_{k+1}&:=P_1+B_1\Delta_{1,k}A_1,\\
Q_{k+1}&:=Q_{k}-A_{k} \Delta_{1,k} B_{k},
  \end{align}
  \end{subequations}
where $\Delta_{1,k}:=(Q_1-P_{k})^{-1}$ and initial matrices $A_1=A$, $B_1=B$, $P_1=0$, and  $Q_1=Q$. This is because
\begin{align*}
&Q_{k+1}-A_{k+1}(X-P_{k+1})^{-1}B_{k+1}=N^{(k+1)}(X) = N^{(k)}(N(X))\\
 &=Q_{k}-A_{k}(Q_1-P_{k}-A_1(X-P_1)^{-1}B_1)^{-1}B_{k}\\%(\mbox{by SMWF})\\
 &=Q_{k}-A_{k}(\Delta_{1,k}+\Delta_{1,k}A_1(X-P_1-B_1\Delta_{1,k}A_1)^{-1}B_1\Delta_{1,k})B_{k}.
 \end{align*}

The reader is referred to~\cite{Chiang2017} for a thorough discussion of a class of NMEs~\eqref{eq:NME}, where $A= B^H$. It can be proved that the limit of
$Q_k$ is exactly equal to a solution of~\eqref{eq:NME} under some additional assumptions on the coefficients $Q$, $A$ and $B$.

Our goal here is to show that the iteration defined by~\eqref{ex4fix} can be advanced by showing that the semigroup property holds for the mapping
  \begin{align}\label{type1NME}
  F(X_{a},X_{b})=\bb A_{b}\Delta_{a,b}A_{a} \\ B_{a}\Delta_{a,b}B_{b} \\P_{a}+B_{a}\Delta_{a,b}A_{a} \\ Q_{b}-A_{b}\Delta_{a,b}B_{b}\eb,
  \end{align}
where
$X_{a}=\bb A_{a} \\ B_{a} \\ P_{a} \\Q_{a} \eb$ and $\Delta_{a,b}=(Q_{a}-P_{b})^{-1}$ for any two dummy indices $a$ and $b$. That is,  for $k\geq 1$,
  \begin{align*}
  X_{k+1}= F(X_{k},X_{1}).
  \end{align*}
To show this desired property, we have the following preliminary result by applying the SMWF.

\begin{Lemma}\label{lem:NME}
Let $X_{d}=F(X_{a},X_{b})$ and $X_{e}=F(X_{b},X_{c})$. Then
\begin{align*}
\begin{array}{ll}
 \rm{(i)} & \Delta_{a,e}=\Delta_{a,b}+\Delta_{a,b}B_b\Delta_{d,c}A_b\Delta_{a,b},\\
 \rm{(ii)} &
 \Delta_{d,c}=\Delta_{b,c}+\Delta_{b,c}A_b\Delta_{a,e}B_b\Delta_{b,c},\\
 \rm{(iii)} & \Delta_{a,b} B_b \Delta_{d,c}=\Delta_{a,e} B_b \Delta_{b,c},\\
 \rm{(iv)}  & \Delta_{d,c} A_b \Delta_{a,b}=\Delta_{b,c} A_b \Delta_{a,e}.
 \end{array}
 \end{align*}

% \item [\rm{(i)}]
%  \begin{align*}
% \Delta_{4,3}&=\Delta_{2,3}+\Delta_{2,3}A^{(2)}\Delta_{1,5}B^{(2)}\Delta_{2,3},\\
% \Delta_{1,5}&=\Delta_{1,2}+\Delta_{1,2}B^{(2)}\Delta_{4,3}A^{(2)}\Delta_{1,2}.
%  \end{align*}
%  \item [(ii)]
%  \begin{align*}
%  \Delta_{1,2} B^{(2)} \Delta_{4,3}&=\Delta_{1,5} B^{(2)} \Delta_{2,3},\\
%  \Delta_{4,3} A^{(2)} \Delta_{1,2}&=\Delta_{2,3} A^{(2)} \Delta_{1,5}.
%  \end{align*}
%\end{itemize}
\end{Lemma}
\begin{proof}
Observe that
\begin{align*}
  \Delta_{a,e}&=(Q_a-P_e)^{-1}=(Q_a-P_b-B_b\Delta_{b,c}A_b)^{-1}\\
 &=\Delta_{a,b}+\Delta_{a,b}B_b(Q_b-P_c-A_b\Delta_{a,b}B_b^{-1}A_b\Delta_{a,b},\\
 &=\Delta_{a,b}+\Delta_{a,b}B_b\Delta_{d,c}A_b\Delta_{a,b},\\
 \Delta_{d,c}&=(Q_d-P_c)^{-1}=(Q_b-P_c-A_b\Delta_{a,b}B_b)^{-1}\\
 &=\Delta_{b,c}+\Delta_{b,c}A_b(Q_a-P_b-B_b\Delta_{b,c}A_b)^{-1}A_b\Delta_{b,c},\\
 &=\Delta_{b,c}+\Delta_{b,c}A_b\Delta_{a,e}B_b\Delta_{b,c},
  \end{align*}
  which complete the proofs of part (i) and (ii).
Since
\begin{align*}
  \Delta_{a,b}^{-1}\Delta_{a,e}&=I_n+B_b\Delta_{b,c} A_b \Delta_{a,e},\\
  \Delta_{a,e}\Delta_{a,b}^{-1}&=I_n+\Delta_{a,e}B_b\Delta_{b,c}A_b,
  \end{align*}
 it follows that
  \begin{align*}
  \Delta_{a,b} B_b \Delta_{d,c}&=\Delta_{a,b}B_b\Delta_{b,c}+\Delta_{a,b}B_b\Delta_{b,c}A_b\Delta_{a,e}B_b\Delta_{b,c}\\
  &=\Delta_{a,b}(I_n+B_b\Delta_{b,c}A_b\Delta_{a,e})B_b\Delta_{b,c}\\
  &=\Delta_{a,e} B_b \Delta_{b,c},\\
  \Delta_{d,c}A_b\Delta_{a,b}&=\Delta_{b,c}A_b\Delta_{a,b}+\Delta_{b,c}A_b\Delta_{a,e}B_b\Delta_{b,c}A_b\Delta_{a,b}\\
  &=\Delta_{b,c}A_b(I_n+\Delta_{a,e}B_b\Delta_{b,c}A_b)\Delta_{a,b}\\
  &=\Delta_{b,c} A_b \Delta_{a,e}.
  \end{align*}

\end{proof}

The new identities established in Lemma~\ref{lem:NME} lead to the following result:
\begin{Theorem}\label{thm:nme}
For the matrix operator \eqref{type1NME}, we have
\begin{align*}
 F(F(X_a,X_b),X_c)=F(X_a,F(X_b,X_c)).
\end{align*}
\end{Theorem}
\begin{proof}
 Let $X_{d}=F(X_{a},X_{b})$, $X_{e}=F(X_{b},X_{c})$, $X_{\ell}=F(X_d,X_{c})$ and $X_{r}=F(X_{a},X_{e})$. The semigroup property follows from the following:
\begin{align*}
 A_\ell
  &=A_c \Delta_{d,c} A_d=A_c \Delta_{d,c} A_b \Delta_{a,b} A_a=A_c (\Delta_{d,c} A_b \Delta_{a,b}) A_a=A_c (\Delta_{b,c} A_b \Delta_{a,e}) A_a=A_e \Delta_{a,e} A_a=A_r,\\
B_\ell&=B_d \Delta_{d,c} B_c=B_a\Delta_{a,b} B_b \Delta_{d,c} B_c =
B_a(\Delta_{a,b} B_b \Delta_{d,c}) B_c
=
B_a (\Delta_{a,e} B_b \Delta_{b,c}) B_c=B_r,\\
P_\ell&=P_d+B_d\Delta_{d,c}A_d=P_a+B_a\Delta_{a,b}A_a+B_a\Delta_{a,b}B_b\Delta_{d,c}A_b\Delta_{a,b}A_a\\
&=P_a+B_a\Delta_{a,e}A_a=P_r,\\
Q_\ell&=Q_c-A_c\Delta_{d,c}B_c=
Q_c-A_c(\Delta_{b,c}+\Delta_{b,c}A_b\Delta_{a,e}B_b\Delta_{b,c})B_c\\
&=Q_e-A_e\Delta_{a,e}B_e=Q_r.
   \end{align*}
\end{proof}
From Theorem~\ref{thm:flow} and Theorem~\ref{thm:nme}, we can apply Algorithm~\ref{aa2} to provide an accelerated method. More details can be found in~\cite{Chiang2017}.

\end{example}

%
%
%Ex5:The discrete-time algebraic Riccati equations
%
%
%

\begin{example}
%{Ex5:The discrete-time algebraic Riccati equations}
Algebraic Riccati equations have been widely discussed in the files of control and engineering problems~\cite{Bini2012,Lancaster95}. In the last example, we would like to make use of the semigroup property to promptly solve the discrete-time algebraic Riccati equation~\cite{Huang2018,Lin2018a},
\begin{align}\label{eq:dare}
X=R(X):=H+A^H {X} (I+ G{X})^{-1} A,
\end{align}
where {$A\in\mathbb{C}^{n\times n}$, matrices $G$ and $H$ are two $n\times n$
 positive definite matrices, and $X$ is an unknown Hermitian matrix and to be determined}. Let $X$ be a solution of~\eqref{eq:dare}.
 It follows from a direct computation that
 \begin{align}\label{eq:dare2}
 X = R^{(k)}(R(X))= R^{(k+1)}(X) = H_{k+1}+A_{k+1}^H {X} (I+ G_{k+1}{X})^{-1} A_{k+1},
 \end{align}
 where $A_k$, $G_k$, and $H_k$, for $k = 1,2,\ldots$, are three matrices given by
\begin{subequations}\label{fixdare}
\begin{eqnarray}
A_{k+1} &=&
A_1\Delta_{G_{k},{H_1}}A_{k},\\
G_{k+1} &=&
G_1+A_1\Delta_{G_{k},{H_1}}G_{k}A_1^H,\\
H_{k+1} &=&
H_{k}+A_{k}^H H_1\Delta_{G_{k},{H_1}}A_{k},\end{eqnarray}
\end{subequations}
with respect to initial matrices $G_1=G$, $H_1=H$, $A_1=A$,
and
$\Delta_{G_{k},{H_1}} = I+G_{k}H_1$.
To see how the semigroup property is included in the iterations of ~\eqref{eq:dare2}, assume that~\eqref{fixdare} is well-defined. For any two dummy indices $a$ and $b$,
let
$X_{a}=\bb A_{a}\\G_{a}\\ H_{a} \eb$, $\Delta_{a,b}=(I+G_{a}H_{b})^{-1}$, and
  \begin{align}\label{type1b}
  F(X_{a},X_{b})=\bb A_{b}\Delta_{a,b}A_{a} \\ G_{b}+A_{b}\Delta_{a,b}G_{a}A_{b}^H \\
  H_{a}+A_{a}^H H_{b}\Delta_{a,b}A_{a}\eb,
  \end{align}
We thus have an equivalent expression of~\eqref{fixdare}, i.e.,
\begin{align*}
  X_{k+1}=F(X_{k},X_{1}), \quad k\geq 1.
  \end{align*}
It can be proved that the limit of
$H_k$ is exactly to equal to a solution of~\eqref{eq:dare} under some mild assumptions on the coefficients $A$, $G$, and $H$~\cite{Lin2018a}.

To show the satisfaction of the semigroup property, we require the following characteristics.
\begin{Lemma}\label{lem:CARE}
Let $X_{d}=F(X_{a},X_{b})$ and $X_{e}=F(X_{b},X_{c})$.  Then\
\begin{align*}
\begin{array}{ll}
 \rm{(i)} & \Delta_{a,e}=\Delta_{a,b}-\Delta_{a,b}(G_aA_b^H H_c)\Delta_{d,c}A_b\Delta_{a,b},\\
  \rm{(ii)} & \Delta_{d,c}=\Delta_{b,c}-\Delta_{b,c}A_b\Delta_{a,e}(G_aA_b^H H_c)\Delta_{b,c},\\
 \rm{(iii)} & \Delta_{a,b} (G_aA_b^H H_c) \Delta_{d,c}=\Delta_{a,e}  (G_aA_b^H H_c) \Delta_{b,c},\\
 \rm{(iv)} &\Delta_{d,c} A_b \Delta_{a,b}=\Delta_{b,c} A_b \Delta_{a,e}.
\end{array}
\end{align*}

\end{Lemma}
\begin{proof}
The proofs of the first two results follow trivially by the SWMF, i.e.,
 \begin{align*}
 \Delta_{a,e}&=(I_n+G_aH_e)^{-1}=(I_n+G_aH_b+G_aA_b^H H_c\Delta_{b,c}A_b)^{-1}\\
 &=\Delta_{a,b}-\Delta_{a,b}G_aA_b^H H_c(I+G_bH_c+A_b{\Delta_{b,c}}G_aA_b^H H_c)^{-1}A_b\Delta_{a,b},\\
 &=\Delta_{a,b}-\Delta_{a,b}(G_aA_b^H H_c)\Delta_{d,c}A_b\Delta_{a,b},\\
 \Delta_{d,c}&=(I_n+G_dH_c)^{-1}=(I_n+G_bH_c+A_b\Delta_{a,b}G_aA_b^H H_c)^{-1}\\
 &=\Delta_{b,c}-\Delta_{b,c}A_b(I_n+G_aH_b+G_aA_b^H H_c\Delta_{b,c}A_b)^{-1}G_aA_b^H H_c\Delta_{b,c},\\
 &=\Delta_{b,c}-\Delta_{b,c}A_b\Delta_{a,e}(G_aA_b^H H_c)\Delta_{b,c}.
\end{align*}
Also, it can be seen that
   \begin{align*}
  \Delta_{a,b}^{-1}\Delta_{a,e}&=I_n-G_a A_b^H H_c\Delta_{b,c} A_b \Delta_{a,e} ,\\
  \Delta_{a,e}\Delta_{a,b}^{-1}&=I_n-\Delta_{a,e}G_a A_b^H H_c\Delta_{b,c} A_b.
  \end{align*}
It follows that
  \begin{align*}
  \Delta_{a,b} [G_a A_b^H H_c] \Delta_{d,c}&=\Delta_{a,b} (G_aA_b^H H_c]\Delta_{b,c}-\Delta_{a,b}(G_aA_b^H H_c)\Delta_{b,c}A_b\Delta_{a,e}(G_aA_b^H H_c)\Delta_{b,c}\\
  &=\Delta_{a,b}\left [(I_n+G_aH_e)-G_aA_b^H H_c\Delta_{b,c}A_b \right] \Delta_{a,e} (G_aA_b^H H_c)\Delta_{b,c}\\
    &{=\Delta_{a,b}\left [I_n+ G_a(H_e-A_b^H H_c\Delta_{b,c}A_b) \right] \Delta_{a,e} (G_aA_b^H H_c)\Delta_{b,c}
    }
    \\
  &=\Delta_{a,e} (G_aA_b^H H_c)  \Delta_{b,c},\\
{ \Delta_{d,c}A_b\Delta_{a,b}}&=\Delta_{b,c}A_b\Delta_{a,b}-\Delta_{b,c}A_b\Delta_{a,e}(G_aA_b^H H_c)\Delta_{b,c}A_b\Delta_{a,b}\\
  &=\Delta_{b,c}A_b(I_n-\Delta_{a,e}(G_aA_b^H H_c)\Delta_{b,c}A_b)\Delta_{a,b}\\
  &=\Delta_{b,c} A_b {\Delta_{a,e}}
  \end{align*}
\end{proof}

The following theorem follows from the application of Lemma~\ref{lem:CARE}.
\begin{Theorem}\label{thm:care}
For the matrix operator \eqref{type1b}, we have
\begin{align*}%\label{eq:smi}
 F(F(X_a,X_b),X_c)=F(X_a,F(X_b,X_c)).
\end{align*}
\end{Theorem}

\begin{proof}
To start with, we define $X_{d}=F(X_{a},X_{b})$, $X_{e}=F(X_{b},X_{c})$, $X_{\ell}=F(X_d,X_{c})$ and $X_{r}=F(X_{a},X_{e})$. We see that
\begin{align*}
  A_{\ell} & =A_c \Delta_{d,c} A_d=A_c \Delta_{d,c} A_b \Delta_{a,b} A_a=A_c (\Delta_{d,c} A_b \Delta_{a,b}) A_a\\
  &=A_c (\Delta_{b,c} A_b \Delta_{a,e}) A_a=A_e \Delta_{a,e} A_a=A_r.
\end{align*}
By letting
\begin{align*}
&\Pi_1=
\Delta_{b,c}A_b\Delta_{a,b}G_aA_b^H-
\Delta_{b,c}A_b\Delta_{a,e}G_a A_b^H H_c\Delta_{b,c}A_b\Delta_{a,b}G_aA_b^H,\\
&\Pi_2=
\Delta_{b,c}A_b\Delta_{a,e}G_a A_b^H H_c\Delta_{b,c}G_b,\\
&\Pi_3 =
A_b^H H_c\Delta_{b,c}A_b\Delta_{a,b}-A_b^H H_c\Delta_{b,c}A_b\Delta_{a,b}G_a A_b^H H_c\Delta_{d,c}A_b\Delta_{a,b},\\
&\Pi_4 =
H_b\Delta_{a,b} G_a A_b^H H_c\Delta_{d,c}A_b\Delta_{a,b},
\end{align*}
we have
\[
\Delta_{d,c}G_d=\Delta_{b,c}G_b+\Pi_1-\Pi_2,
\]
where
\begin{align*}
\Pi_1
&=\Delta_{b,c}A_b(I_n-\Delta_{a,e}G_a A_b^H H_c\Delta_{b,c}A_b)\Delta_{a,b}G_a A_b^H\\
&=\Delta_{b,c}A_b\Delta_{a,e} [I_n+G_a(H_e-A_b^H H_c\Delta_{b,c}A_b)]\Delta_{a,b}G_a A_b^H\\
&=\Delta_{b,c}A_b\Delta_{a,e}G_a A_b^H
\end{align*}
and
\begin{align*}
\Pi_1-\Pi_2
&=\Delta_{b,c}A_b\Delta_{a,e}G_a A_b^H(I_n-H_c\Delta_{b,c}G_b)\\
&=\Delta_{b,c}A_b\Delta_{a,e}G_a A_b^H\Delta_{c,b}.
\end{align*}
Also,
\[
H_e\Delta_{a,e}=H_b\Delta_{a,b}+\Psi_1-\Psi_2,
\]
where
\begin{align*}
\Psi_1
&=A_b^H H_c\Delta_{b,c}(I_n-A_b\Delta_{a,b}G_b A_b^H H_c\Delta_{d,c})A_b\Delta_{a,b}\\
&=A_b^H H_c\Delta_{b,c}
{%\color{blue}
[(I_n+G_dH_c)-A_b\Delta_{a,b}G_a A_b^H H_c]\Delta_{d,c}A_b\Delta_{a,b}
}
\\
&=A_b^H H_c\Delta_{b,c}
{
[I_n+(G_d-A_b\Delta_{a,b}G_a A_b^H) H_c]\Delta_{d,c}A_b\Delta_{a,b}
}
\\
&=A_b^H H_c\Delta_{b,c}
{
(I_n+G_bH_c)\Delta_{d,c}A_b\Delta_{a,b}
}
\\
&=A_b^H H_c
{
\Delta_{d,c}A_b
}
\Delta_{a,b}
\end{align*}
and
\begin{align*}
\Psi_1 - \Psi_2
&={
(I_n-H_b\Delta_{a,b} G_a)
}
A_b^H H_c\Delta_{d,c}A_b\Delta_{a,b}\\
&={%\color{red}
(I_n+H_bG_a)^{-1}
}
A_b^H H_c\Delta_{d,c}A_b\Delta_{a,b}\\
&={%\color{red}
\Delta_{a,b}^H
}
A_b^H H_c\Delta_{d,c}A_b\Delta_{a,b}.
\end{align*}

Consequently an appropriate calculation is implemented such that
\begin{align*}
G_{\ell}&=G_c+A_c
{
\Delta_{d,c}G_d
}
A_c^H\\
       &=G_c+A_c(\Delta_{b,c}G_b+\Delta_{b,c} A_b\Delta_{a,e}G_a
       {
       A_b^H \Delta_{c,b}) A_c^H
       }
       \\
       &=G_c+A_c(\Delta_{b,c}G_b)A_c^H+A_c (\Delta_{b,c} A_b\Delta_{a,e}G_aA_b^H \Delta_{c,b}) A_c^H\\
       &=G_e+A_e\Delta_{a,e}G_a A_e^H=G_r
\end{align*}
and
\begin{align*}
H_{\ell}&=H_d+A_d^H H_c\Delta_{d,c}A_d\\
       &=H_a+A_a^H (H_b\Delta_{a,b})A_a+A_a^H({%\color{red}
       \Delta_{a,b}^H}A_b^H H_c\Delta_{d,c}A_b \Delta_{a,b}] A_a\\
       &=H_a+A_a^H (H_b\Delta_{a,b}+\Delta_{a,b}^HA_b^H H_c\Delta_{d,c}A_b \Delta_{a,b}) A_a\\
       &=H_a+A_a^H
       {%\color{red}
       H_e\Delta_{a,e}
       }
       A_a=H_r,
\end{align*}
which completes the proof.

\end{proof}
Theorem~\ref{thm:care} shows that the sequence~\eqref{fixdare} has the semigroup property so that Algorithm~\ref{aa2} can be applied to provide an accelerated method. See~\cite{Lin2018a} for more details.

\end{example}

\section{Concluding remark}
As is well known, problems in determining the solutions of a matrix equation are firmly related to a wide range of challenging scientific areas.
Traditional approaches for finding a numerical solution are based on the fixed-point iteration, and the speed of the convergence is usually linear. In this paper, we investigate the semigroup property for some binary matrix operations and apply this property to construct one type of iterations for solving several matrix equations while the desired speed of convergence is given.  More precisely, we show that once the existence of the property holds for a given iteration, we can immediately construct an iterative method which can converge to the solution with the speed of convergence of any desired order. To interpret the robustness and capacity of our method, we apply this property to analyze some structured iterations arising from a class of matrix equations

For future work, we choose to treat some interesting problems. The first one is that for a given matrix equation,  how to discover the kind of iteration $X_{k+1}=F(X_k,X_1)$ so that the semigroup property holds. The other challenge problem is that how to simfplify the accelerated process for a fixed-point iteration if the semigroup property holds? All these questions are under investigation and will be reported elsewhere.
\section*{Acknowledgment}
This research work is partially supported by the Ministry of Science and Technology and the National Center for Theoretical Sciences in Taiwan. The first author (Matthew M. Lin) would like to thank the support from the Ministry of Science and Technology of Taiwan
under grants MOST 107-2115-M-006-007-MY2, and the corresponding author (Chun-Yueh Chiang) would like to thank the support from the Ministry of Science and Technology of Taiwan under the grant MOST 107-2115-M-150-002.

\bibliographystyle{plain}
%\bibliographystyle{abbrv}
%\bibliography{chiang_new}

\end{document}